\documentclass[letterpaper]{article} 
\usepackage{aaai24}  
\usepackage{times}  
\usepackage{helvet}  
\usepackage{courier}  
\usepackage[hyphens]{url}  
\usepackage{graphicx} 
\usepackage{natbib}  
\usepackage{caption} 
\usepackage{algorithm}
\usepackage{algpseudocode}

\usepackage{newfloat}
\usepackage{listings}
\DeclareCaptionStyle{ruled}{labelfont=normalfont,labelsep=colon,strut=off} 
\floatstyle{ruled}
\newfloat{listing}{tb}{lst}{}
\floatname{listing}{Listing}
\usepackage{amsmath}
\usepackage{amssymb}
\usepackage{mathtools}
\usepackage{amsthm}
\usepackage{blindtext}

\usepackage{xcolor}         
\definecolor{dgreen}{rgb}{0, 0.5, 0}
\definecolor{orange}{rgb}{1, 0.37, 0.12}

\newif
\ifdraft
\draftfalse

\ifdraft
\usepackage{lipsum}
\usepackage[notref,notcite,color]{showkeys}
\usepackage[cam,width=10in,height=12in,center]{crop}
\definecolor{refkey}{rgb}{0,.6,0}
\definecolor{labelkey}{cmyk}{0, 1, 0, 0}
\fi

\def\R{\mathbb{R}}
\def\eps{\varepsilon}
\DeclareMathOperator*{\argmin}{arg\,min}
\def\state{x} 
\def\costate{p} 
\def\control{\alpha} 
\def\controlbis{\beta} 
\def\time{t}
\def\statevar{\xi} 
\def\costatevar{\eta} 
\def\controlvar{a} 
\def\timevar{s}
\def\inputvar{\eta}
\def\f{f}
\def\inp{u}
\def\target{z}
\def\inpdim{d}
\def\BV{\mathit{BV}}
\def\statenet{\gamma}
\def\cost{C}
\def\loss{\ell}
\def\controlspace{\mathcal{A}}
\def\valuef{v}
\def\costate{p}
\def\costatevar{\rho}
\def\graph{G}
\def\vertices{V}
\def\arcs{E}
\def\one{\mathop{1}\nolimits}
\def\costatenet{\mu}
\def\mudim{M}
\def\thetavar{\vartheta}
\def\neuron{y}
\def\neurondim{m}
\def\rev#1{#1}

\def\ch{\mathop{\rm ch}\nolimits}
\def\pa{\mathop{\rm pa}\nolimits}

\def\statedim{n} 
\def\controldim{N} 
\def\outdim{p} 

\theoremstyle{plain}

\newtheorem{theorem}{Theorem}

\theoremstyle{definition}
\newtheorem{definition}{Definition}
\newtheorem{example}{Example}

\theoremstyle{remark}
\newtheorem{remark}{Remark}

\title{Neural Time-Reversed Generalized Riccati Equation}
\author{
    Alessandro Betti\textsuperscript{\rm 1},
    Michele Casoni\textsuperscript{\rm 2},
    Marco Gori\textsuperscript{\rm 1,2},
    Simone Marullo\textsuperscript{\rm 2,3},\\
    Stefano Melacci\textsuperscript{\rm 2},
    Matteo Tiezzi\textsuperscript{\rm 2}
}
\affiliations{
    \textsuperscript{\rm 1} Inria, Lab I3S, MAASAI, Universit\`e C\^ote d'Azur, Nice, France\\
    \texttt{alessandro.betti@inria.fr}\\
    \textsuperscript{\rm 2}  DIISM, University of Siena, Siena, Italy\\
    \texttt{m.casoni@student.unisi.it}, \texttt{\{marco,mela, mtiezzi\}@diism.unisi.it}\\
    \textsuperscript{\rm 3}  DINFO, University of Florence, Florence, Italy\\
    \texttt{simone.marullo@unifi.it}
}

\begin{document}

\maketitle

\begin{abstract}\rev{
Optimal control deals with optimization problems in which variables steer a
dynamical system, and its outcome contributes to the objective function. Two
classical approaches to solving these problems are \emph{Dynamic Programming}
and the \emph{Pontryagin Maximum Principle}. In both approaches, Hamiltonian
equations offer an interpretation of optimality through auxiliary variables
known as \emph{costates}. However, Hamiltonian equations are rarely used due
to their reliance on forward-backward algorithms across the entire temporal
domain.  This paper introduces a novel neural-based approach to optimal
control, with the aim of working forward-in-time. Neural networks are employed not only for implementing state
dynamics but also for estimating costate variables. The parameters of the
latter network are determined at each time step using a newly introduced
\emph{local} policy referred to as the \emph{time-reversed generalized
Riccati equation}. This policy is inspired by a result discussed in the
Linear Quadratic (LQ) problem, which we conjecture stabilizes state
dynamics. We support this conjecture by discussing experimental results from
a range of optimal control case studies.
}
\end{abstract}

\section{Introduction}
\label{sec:intro}
Optimal control \cite{lewis2012optimal} offers a wide framework to set up optimization problems that are
concerned with the steering of a dynamical system in some parsimonious
way. It is therefore clear that its scope is quite large and it
intersects many areas such as, for instance, pure math, natural sciences and engineering.
Being the optimization problem objective defined on the solution of a system
of ODEs over a certain temporal horizon $[\time_0,T]$, it has a
global-in-time nature.  Indeed, classical approaches to optimal control such
as the \emph{Pontryagin Maximum Principle}
(see~\cite{gamkrelidze1964mathematical,giaquinta2013calculus})
and \emph{dynamic programming} (see~\cite{bardi1997optimal})
both characterize solutions
in terms of a boundary problem for some differential conditions (usually a
PDE in dynamic programming and a system of ODEs with the Pontryagin maximum
principle). This means that, in general, the algorithms to find solutions
require iterative forward/backward approaches to glue the local-in-time
computations of the differential equations with the boundary conditions at
opposite sides of the temporal interval.

In many instances of control problems, where either the complexity of the
model is high and/or the temporal horizon could be very long, as it could
happen for instance in Reinforcement Learning~\cite{sutton2018reinforcement,
bertsekas2019reinforcement} or Lifelong Learning~\cite{betti2022continual,
MAI202228}, these methods are unfeasible and we usually need to resort to
different control strategies.  A typical approach is that of using
\emph{Model Predictive Control}~\cite{garcia1989model} (also known as
\emph{receding horizon control}), with a real-time iteration (RTI) scheme for
solving the online optimization problem~\cite{doi:10.1137/S0363012902400713}.
The necessity of finding optimization procedures that only exploit
forward (in time) computations is an especially sensible matter within
the machine learning community, where the possibility of performing a
backpropagation through the entire temporal horizon (backpropagation through
time) is considered to be extremely implausible from a biological point
of view~\cite{hinton2022forward} and in some cases prohibitively costly.

In this work we present a novel approach that makes use of Hamilton
Equations giving an estimate of the costate function through a
neural-network computation, working \textit{forward-in-time}.  The basic idea of our approach is to estimate
the parameters of this network by means on an indirect usage of the Hamilton equations. \rev{Recently, in \cite{Jin2019PontryaginDP}, the possibility of exploiting Hamiltonian equations for learning system dynamics and controlling policies forward in time has been investigated. In this paper, the authors introduced {\it Pontryagin Differentiable Programming} (PDP) to efficiently compute the gradients of the state trajectory with respect to the system parameters using an auxiliary control system.  This approach differs from the method proposed in this paper by the fact that, 
instead, in the present work, we use  Hamilton equations
indirectly for defining an optimization problem for the temporal variations of the model parameters.}
In doing so, we are basically defining a dynamic on the
parameters that estimate the costate in a similar manner as we would do with
the Riccati equation in the Linear Quadratic control problem.  We conjecture
that the resulting time-reversed dynamics will lead to a stabilizing effect
on the state equation, hence opening the possibility to use this method
forward-in-time. 

This approach has been inspired by the possibility of using optimal
control techniques in the continual online learning scenario recently
proposed in~\cite{betti2022continual} to formulate a class of
lifelong problems using the formalism of control theory.
For this reason, throughout the paper we assume that the dynamical system
that defines the evolution of the state is also expressed by a neural
model, in the form of a continuous time recurrent neural
network~\cite{6814892}.
The authors in~\cite{betti2022continual}
proposed a method to enforce stability by pushing the costate dynamic to
converge to zero, and hence directly interfering with the dynamics prescribed
by Hamilton equations. Conversely, we directly leverage on
Hamilton equations to devise a stabilizing policy for the system.

The paper is organized as follows. In Section~\ref{sec:cont} we describe the class of
dynamical models that we take into account, Section~\ref{sec:ocp} is devoted to the
formulation of the control problem and contains a short review of the main
results from optimal control that will be used in the reminder of the
paper. Section~\ref{sec:flip} constitutes the core part of the contribution and it is 
where we introduce the time-reversed generalized Riccati equation.
Section~\ref{sec:exp} contains the experimental observations that have been
organized in three different case studies, while conclusions and ideas for future work are the subjects of Section~\ref{sec:lim}.

\section{Continuous time state model}
\label{sec:cont}
Let us focus on models that depend on $\controldim$ parameters, whose values at time $\time$ are yielded by $\control(\time)$, and that are based on an internal state $\state(\time)$ of size $\statedim$ which dynamically changes over time. 
We consider a classic state model 
\begin{equation}\label{eq:state-model}
\state'(\time)= \f(\state(\time),\control(\time),\time),\quad t\in(\time_0,T]
\end{equation}
where $\f\colon\R^\statedim\times\R^\controldim\times[\time_0,T]\to\R^\statedim$
is a Lipschitz function, $t\mapsto\control(t)$ is the trajectory of the parameters 
of the model, which is assumed to be a \emph{measurable function}, and
$T$ is the temporal horizon on which the model is defined; $\time_0\ge0$.
We assume that the $\outdim$-components output of the model is computed by a \emph{fixed}
transformation of the state, $\pi\colon \R^\statedim\to\R^\outdim$, usually a
projection of class $C^\infty(\R^\statedim;\R^\outdim)$.
The initial state of the model is assigned to a fixed vector $\state^0\in
\R^\statedim$, that is
\begin{equation}\label{eq:state-initialization}
\state(\time_0)=\state^0.
\end{equation}
Let us now pose $\controlspace:=\{\control\colon[\time_0,T]\to\R^\controldim:
\control\hbox{ is measurable}\}$. 
\begin{definition}\label{def:solution}
Given a $\controlbis\in\controlspace$, and given an initial state $\state^0$,
we define the \emph{state trajectory}, that we indicate with
$t\mapsto x(t;\controlbis,\state^0,\time_0)$, the solution of~\eqref{eq:state-model} with
initial condition \eqref{eq:state-initialization}.
\label{def1}
\end{definition}


\rev{The goal of this work is to define a procedure to estimate
with a {\it forward-in-time} scheme an approximation of the optimal  control parameters $\control$.\footnote{The meaning of optimality
will be described in details in Section~\ref{sec:ocp}}  
Notice that the explicit time dependence $\time$ of Eq.~\eqref{eq:state-model} is necessary to take into account 
the provision over time of some input data to the model.
In the next section, we will give a more precise structure to such
temporal dependence.}

\subsection{Neural state model}\label{sec:state-model}
We implement the function $\f$ of Eq.~\eqref{eq:state-model} by a neural network $\statenet$, where the dependence on time $\time$ is indirectly modeled by a novel function $\inp(\time)$, that yields the $\inpdim$-dimensional input data provided at time $\time$ to the network. 
Formally, for all $\statevar\in\R^\statedim$, for all $\timevar\in[\time_0,T]$ and all $\controlvar\in\R^\controldim$,
$\f(\statevar,\controlvar,\timevar):=\statenet(\statevar,\inp(\timevar),\controlvar)$, where,
for all fixed $\controlvar\in\R^\controldim$, the map $\statenet(\cdot,\cdot,\controlvar)
\colon\R^\statedim\times\R^\inpdim\to\R^\statedim$ is a neural network and
 $u\colon[\time_0,T]\to\R^\inpdim$ is the input signal, being
$u\in\BV((\time_0,T))$ an assigned input map of bounded variation\footnote{
Here the space $\BV(\time_0,T)$ is the functional space of functions of bounded variation,
see~\cite{ambrosio2000functions}.}.
More directly we can assume that we are dealing with a
Continuous Time Recurrent Neural Network (CTRNN) (see~\cite{6814892}) that at each instant estimates the
variation of the state based on the current value of the state itself and on an
external input. The network $\statenet(\cdot,\cdot,\controlvar)$ represents the
transition function of the state. In this new notation, the dynamic of the
state, given by Eq.~\eqref{eq:state-model} together with Eq.~\eqref{eq:state-initialization},
is described by the following Cauchy problem
for $\state$:
\begin{equation}\label{eq:state-model-nn}
\begin{cases}
\state'(\time)=\statenet(\state(\time), \inp(\time),\control(\time)), \quad \hbox{for $\time\in(\time_0,T]$} \\
\state(\time_0)=x^0.
\end{cases}
\end{equation}
To help the reader in giving an initial interpretation to the  parameters $\alpha$, at this stage it is enough to assume that 
 $\control$ could basically represent the weights and the
biases of the network $\statenet$. Similarly, the state $\state$ could be imagined as the usual state in a CTRNN. However, there are still some steps to take before providing both $\control$ and $\state$ the exact role we have considered in this paper. 
First, we need to define the way $\control$ participates in an optimization problem, defining a {\it control problem} 
whose {\it control parameters} are $\control$. 
This will be the main topic of the next Section~\ref{sec:ocp}, where we will start from the generic state model of the beginning of Section~\ref{sec:cont}, and then cast the descriptions on the neural state model $\statenet$---Section~\ref{sec:ocp-net}.
When doing it, we will also reconsider the role of $\control$ in the context of the neural network $\statenet$, due to some requirements introduced by the optimization procedure over time.

\section{Control problem}\label{sec:ocp}
Suppose now that we want to use the model described in Eq.~\eqref{eq:state-model} paired with Eq.~\eqref{eq:state-initialization} to solve some task that can be expressed as a minimization problem
for a cost functional $\control\mapsto C(\alpha)$.
We recall the notation $\state(\time;\control,\state^0,\time_0)$, introduced in Def.~\eqref{def1},
to compactly indicate the state $\state$ and all its dependencies as a solution of
Eq.~\eqref{eq:state-model} with initial values~\eqref{eq:state-initialization}.
The cost functional has the following form:
\begin{equation}\label{eq:cost-functional}
\cost_{\state^0, \time_0}(\control):=
\int_{\time_0}^T \loss(\control(\time), \state(\time;\control,\state^0,\time_0), \time)\, d\time,
\end{equation}
where $\loss(\controlvar,\cdot,\timevar)$ is bounded and Lipshitz $\forall \controlvar\in\R^\controldim$ and
$\forall\timevar\in[\time_0,T]$.
The function $\loss$ is usually called \emph{Lagrangian} and it can be thought as the counterpart of a classic machine-learning loss function in control theory.
Because the term $\state(\time;\control,\state^0,\time_0)$ in Eq.~\eqref{eq:cost-functional}
depends on the variables $\control$ through the integration of a first-order
dynamical system, the problem
\begin{equation}\label{eq:minimization-of-cost}
\min_{\control\in\controlspace} \cost_{\state^0,\time_0}(\control)
\end{equation}
is  a constrained minimization problem which is usually denoted
as \emph{control problem}~\cite{bardi1997optimal}, assuming that a solution exists. 

A classical way to address problem \eqref{eq:minimization-of-cost} is through dynamic 
programming and the Hamilton-Jacobi-Bellman equation~\cite{bardi1997optimal},
that will be the key approach on which we will build the ideas of this paper, paired with some intuitions to yield a forward solution over time. 
We briefly summarize such classical approach in the following.
The first step to address our constrained minimization problem is to define the \emph{value function} or \emph{cost to go}, that is a 
map $\valuef\colon\R^\statedim\times[\time_0,T]\to\R$ defined as 
\[\valuef(\statevar,\timevar):= 
\inf_{\alpha\in\controlspace} C_{\statevar,\timevar}(\alpha),\quad \forall(\statevar,
\timevar)
\in\R^\statedim\times[\time_0,T].\]
The optimality condition of the cost $\cost$ then translates into an 
infinitesimal condition (PDE) for the value function $\valuef$ (see~\cite{bardi1997optimal});
such result can be more succinctly stated once we define the \emph{Hamiltonian} function 
$H\colon\R^\statedim\times\R^\statedim\times[\time_0,T]\to \R$
\begin{equation}\label{eq:hamiltonian}
H(\statevar,\costatevar,\timevar):=\min_{\controlvar\in\R^\controldim} 
\{\costatevar\cdot f(\statevar ,\controlvar,\timevar)+\ell(\controlvar,\statevar,\timevar)\},
\end{equation}
being $\cdot$ the dot product.
Then the following well-known result holds.
\begin{theorem}[Hamilton-Jacobi-Bellman] \label{th:HJB}
Let us assume that $D$ denotes the gradient operator with respect to $\statevar$.
Furthermore, let us assume that $\valuef\in
C^1(\R^\statedim\times[\time_0,T],\R)$ and that the minimum of $C_{\statevar,\timevar}$, Eq.~\eqref{eq:minimization-of-cost}, exists for
every $\statevar\in\R^\statedim$ and for every $\timevar\in[\time_0,T]$. Then
$\valuef$ solves the PDE
\begin{equation}\label{eq:HJB}
v_\timevar(\statevar,\timevar)+H(\statevar,Dv(\statevar,\timevar),\timevar)=0, 
\end{equation}
$(\statevar,\timevar)\in \R^\statedim\times[\time_0,T)$, with terminal condition $\valuef(\statevar,T)=0$, $\forall \statevar\in\R^\statedim$.
Equation~\eqref{eq:HJB} is usually referred to as Hamilton-Jacobi-Bellman equation.
\end{theorem}
\begin{proof}
See appendix~A.
\end{proof}

The result stated in Theorem~\ref{th:HJB} gives a characterization of the value function;
the knowledge of the value function in turn gives a direct way to construct a
solution of the problem defined in Eq.~\eqref{eq:minimization-of-cost} by a standard
procedure called \emph{synthesis procedure}~\cite{evans2022partial,bardi1997optimal}, for which we summarize its main ingredients.
The first step is, once a solution of Eq.~\eqref{eq:HJB} with the terminal condition
$\valuef(\statevar,T)=0$, $\forall \statevar\in\R^\statedim$ is known,
to find an \emph{optimal feedback map}
$S\colon\R^\statedim\times[\time_0,T]\to\R^\controldim$ defined by the condition
\begin{equation}\label{eq:feedback}
S(\statevar,\timevar)\in\argmin_{\controlvar\in\R^\controldim}
\{D\valuef(\statevar,\timevar)\cdot
f(\statevar ,\controlvar,\timevar)+\ell(\controlvar,\statevar,\timevar)\}.
\end{equation}
Once a function $S$ with such property is computed, the second step is to solve
\[
\state'(\time)=f(\state(\time),S(\state(\time),\time),\time), \quad\hbox{for}\quad\time\in(\time_0,T), \]
with initial condition $\state(\time_0)=\state^0$, and call a solution of this equation $\state^*$. Then the optimal control $\control^*$ is
directly given by the feedback map:
\begin{equation}
    \alpha^*(\time)=S(\state^*(\time),\time).
\label{eq:star}
\end{equation}


\paragraph{Hamilton equations}
There exists another route that can be followed to face the problem of 
Eq.~\eqref{eq:minimization-of-cost}, and that does not directly make
use of Hamilton-Jacobi-Bellman equation~\eqref{eq:HJB}. Such a route, that we will exploit in the rest of the paper, mainly rely on an alternative
representation of the value function which is obtained through the
\emph{the method of characteristics}~\cite{courant2008methods}
and which basically makes it possible to compute the solution of Hamilton-Jacobi-Bellman
equation along a family of curves that satisfy a set of \emph{ordinary differential equations} (ODEs).
This approach is also equivalent (see~\cite{bardi1997optimal}) to the Pontryagin Maximum Principle~\cite{giaquinta2013calculus}.
Let us define the \emph{costate} $\costate(\time):=D\valuef(\state(\time),\time)$
and consider the following system of ODEs known as \emph{Hamilton Equations},
\begin{equation}\label{eq:H-J-equation}
\begin{cases}
\state'(\time)=H_\costatevar(\state(\time),\costate(\time),\time); &\time\in(\time_0,T]\\
\costate'(\time)=-H_\statevar(\state(\time),\costate(\time),\time); &\time\in(\time_0,T] \\
\state(\time_0)=\state^0;\\
\costate(T)=0,
\end{cases}
\end{equation}
being $H_{\costatevar}$ and $H_{\statevar}$ the derivatives of $H$ with respect to its second and first argument, respectively.
Given a solution to Eq.~\eqref{eq:H-J-equation}, we can find a solution of
Eq.~\eqref{eq:HJB} with the appropriate terminal conditions (see~\cite{bardi1997optimal}).
More importantly, this means that instead of directly finding the value function $v$,
in order to find the optimal control $\control^*$ we can solve Eq.~\eqref{eq:H-J-equation}
to find $\costate^*$ and $\state^*$ and then, as we describe in
Eq.~\eqref{eq:star} and~\eqref{eq:feedback}, choose
\begin{equation}\label{eq:syntesis-costate}
\control^*\in \argmin_{\controlvar\in\controlspace} \{\costate^*(\time)\cdot
f(\state^*(\time)
,\controlvar,\time)+\ell(\controlvar,\state^*(\time),\timevar)\}.
\end{equation}
While the problem reformulated in this way appears to be significantly more
tractable, having traded a PDE for a system of ODEs, the inherent difficulty
of solving a global-in-time optimization problem remains and can be understood
as soon as one realizes that Eq.~\eqref{eq:H-J-equation} is a problem with both initial
and terminal boundary conditions. From a numerical point of view, this means that in
general an iterative procedure over the whole temporal interval is needed (for instance
shooting methods~\cite{osborne1969shooting}), making this approach, based on
Hamilton equations, unfeasible for a large class of problems when the dimension of the
state and/or the length of the temporal interval is big. Finding a forward approach to deal with this issue will be the subject of Section~\ref{sec:flip}, while our next immediate goal is to bridge the just introduced notions to formalize the role of $\control$ in our neural-based approach.

\subsection{Controlling the parameters of the network}\label{sec:ocp-net}
Let us now bridge the just described notions with the neural-based implementation $\statenet$ of the state model, as we have already discussed in Section~\ref{sec:state-model}. In this section, we will give a detailed description of the state variable $\state(\time)$ associated to the neural computation
and we will discuss the specific instance of the controls $\control(t)$ we consider, that are both related to the parameters of net $\statenet$.

We consider a digraph $\graph=(\vertices,\arcs)$ and, without loss of generality,
let us assume that $\vertices=\{1,\dots,\neurondim\}$. Remember that given a
digraph, for each $i\in \vertices$ we can always define two sets $\ch(i):=\{
j\in \vertices : (i,j)\in \arcs\}$ and $\pa(i):=\{j\in \vertices :
(j,i)\in \arcs\}$. The digraph becomes a network as soon as we decorate 
each arc $(j,i)\in\arcs$ with a weight $w_{ij}(\time)$ and each vertex $i\in\vertices$
with a neuron output $\neuron_i(\time)$ and a bias $b_i(\time)$,
for every temporal instant $\time\in[\time_0,T]$. Then the typical
CTRNN computation can be written as
\begin{equation}
\begin{split}
\neuron_i'(t) =-\neuron_i(t) + \sigma\Bigl(& \sum_{j\in\pa(i)}w_{ij}(\time)\neuron_j(\time)
+b_i(\time) \\
&\qquad\qquad {}+\sum_{j=1}^\inpdim k_{ij}(\time)\inp_j(\time)\Bigr),
\end{split}
\label{eq:residual}
\end{equation}
where $k_{ij}(\time)$ is a component of the weight matrix associated with the input $\inp$ and 
$\sigma\colon\R\to\R$ is an activation function. 
In general, when dealing with optimization over time, we want to be able to impose a regularization not only on the
values of the parameters of the network, but also on their temporal variations;
as a matter of fact, in many applications one would consider preferable
slow variations or, if the optimization fully converged, even constant parameters of the network.
For this reason it is convenient to associate the control variables $\control$
with the temporal variations (derivatives) of the network's
parameters. The classic learnable parameters (weights and biases) of the network can be considered as 
part of the state $\state$, together with the neuron outputs
$\neuron$. This requires ($i.$) to extend the neural state model of Eq.~\eqref{eq:state-model-nn}, in order to provide a dynamic to the newly introduced state components, and ($ii.$) to take into account the novel definition of $\control$.
Formally, the state at time $\time$ becomes
$\state(\time) = (\neuron(\time), w(\time), b(\time), k(\time))$
and $\statenet$ is only responsible of computing the dynamic of the $\neuron$-portion of it.
\footnote{We have overloaded the symbol $\statenet$: in Eq.~\eqref{eq:state-model-nn}
was defined as the transition function of the whole state, here only of the $\neuron$ part.}
The state model of Eq.~\eqref{eq:state-model-nn}, involving all the components of $x$ above, is then, 
\begin{equation}\label{eq:full-state-model}
\begin{cases}
\neuron'(\time) = \statenet(\neuron(\time), \inp(\time),w(\time),b(\time),k(\time)) \\
w'_{ij}(\time)=\omega_{ij}(\time), \quad (j,i)\in\arcs\\
b'_i(\time)=\nu_i(\time), \quad  i\in\vertices\\
k'_{ij}(\time)=\chi_{ij}(\time), \quad i\in\vertices \hbox{ and }
j=1,\dots,\inpdim.
\end{cases}
\end{equation}
We can finally formalize the control variables\footnote{To avoid a cumbersome notation
we will denote with the name of a state variable or control variable
without specifying any index simply the list of those variables.}
$\control(\time)=(\omega(\time), \nu(\time), \chi(\time))$,
that, when paired with the previous system of equations, allows us to view such a system as a neural state model in the form $\state'(\time)=f(\state(\time),\control(t),t)$, coherently with Eq.~\eqref{eq:state-model}
and~\eqref{eq:state-model-nn}.
Due to the definition of $\control$, a quadratic penalization in $\control$ amounts to
a penalization on the ``velocities'' of the parameters of the network. 

\section{Time-Reversed generalized Riccati equation}
\label{sec:flip}
As we briefly discussed in Section~\ref{sec:ocp}, the approach to 
problem~\eqref{eq:minimization-of-cost} based on
Hamilton equations is not usually computationally feasible, mostly due
to the fact that it involves boundary conditions on both temporal extrema
$\time_0$ and $T$.  More dramatically, Hamilton equations are not
generally stable. Consider, for instance, the following example
of a widely known control problem:

\begin{example}[Linear Quadratic Problem]\label{ex:sLQ}
The Scalar Linear Quadratic problem
is obtained by choosing $f(\statevar,\controlvar,\timevar)
=A\statevar+B\controlvar$ and $\loss(\controlvar,\statevar,\timevar)=
Q\statevar^2/2+ R\controlvar^2/2$ with $Q$ and $R$ positive and $A\in\R$,
$B\in\R$. In this specific case, it turns out that the Hamiltonian can be computed in closed form:
$H(\statevar,\costatevar,\timevar)=Q\statevar^2/2-B^2\costatevar^2/(2R)+
A\statevar\costatevar$. Hence, the Hamilton equations of Eq.~\eqref{eq:H-J-equation} become
$\state'(\time)=-B^2\costate(\time)/R+A \state(\time)$ and $\costate'(\time)=
-Q\state(\time)-A\costate(\time)$. The solution of such system, for general
initial conditions, have positive exponential modes $\exp(\omega\time)$ with
$\omega=\sqrt{A^2+B^2Q/R}$, that obviously generates instabilities.
\end{example}
However, it turns out that the LQ problem of Example~\ref{ex:sLQ} can be
approached with a novel solution strategy, which yields stability and is the key element we propose and exploit in this paper to motivate our novel
approach to forward-only optimization in neural nets.
We can in fact assume that the costate is estimated by 
$\costate(\time)=\costatenet(\state(\time),\theta(\time))$, that is defined 
by $\mu(\statevar,\thetavar)=\thetavar\statevar$, $\forall(\statevar,\thetavar)\in\R^2$. \rev{In other words, in this example at each time instant
$t$ the costate $p(t)$ is a 
linear function of the state $\state(\time)$ with parameter
$\theta(\time)$. By the definition of the costate, this is equivalent to
assume that the value function $\valuef$ is a quadratic function
of the state.} We then
proceed as follows:
\begin{enumerate}
\item We randomly initialize $\theta(0)$ and set $\state(0)=\state^0$;

\item At a generic temporal instant $\time$, 
under the assumption that $\costate(\time)=\mu(\state(\time),
\theta(\time))$, we consider the condition $\costate'(\time)=
d\costatenet(\state(\time),\theta(\time))/dt$ with $\costate'$
computed with the LQ Hamilton equation \eqref{eq:H-J-equation}:
\[\begin{aligned}
&\costatenet_\statevar(\state(\time),\theta(t))\cdot\state'(\time)
+\costatenet_\thetavar(\state(\time),\theta(t))\cdot\theta'(\time)\\
=&-Q\state(\time) -A\theta(\time)\state(\time).\end{aligned}\]
Solving this for $\theta'(\time)$ we obtain the Riccati equation:
$\theta'(\time)= (B^2/R)\theta^2(\time)-2S\theta(\time)-Q$;

\item We change the sign of the temporal derivative in the Riccati equation
\begin{equation}\label{eq:time-reverse-riccati}
\theta'(\time)= -(B^2/R)\theta^2(\time)+2S\theta(\time)+Q,\end{equation}
and we use it with initial conditions to compute $\time\mapsto\theta(\time)$;

\item Finally, we compute the control parameter using Eq.~\eqref{eq:syntesis-costate},
where the optimal costate is replaced with its estimation given
with the network $\costatenet$.
\end{enumerate}

As it is known, Riccati equation must be solved
with terminal conditions; in our case, since we do not
have any terminal cost, the optimal solution would be recovered imposing the boundary
condition $\theta(T)=0$. Solving this equation with initial conditions,
however, does not have any interpretation in terms of the optimization problem
(differently from the forward solution of the costate in
Hamilton equations).
Instead, let us set for simplicity $\time_0=0$ and define
$\Phi\colon \time\in [0,T]\to s\in[0,T]$, with $s:=T-\time$, so that if we let
 $\hat\theta:= \theta\circ\Phi^{-1}$, we have
$\forall s\in[0,T]$ that $\hat \theta(s)= \theta(\Phi^{-1}(s))=
\theta(T-s)$. This time, $\hat \theta$ will satisfy exactly 
Eq.~\eqref{eq:time-reverse-riccati}.
The solution of this equation with \emph{initial condition}
$\hat\theta(0)=0$ can be found explicitly by standard techniques:
\[\hat \theta(s)=\frac{R}{B^2}\lambda_1\lambda_2\frac{
e^{\lambda_1s} -e^{\lambda_2s}}{
\lambda_2e^{\lambda_1s} -\lambda_1 e^{\lambda_2s}
},\]
with
\[\lambda_{1,2}=A\pm\sqrt{A^2+\frac{QB^2}{R}}.\]
This solution has the interesting property that as
$T\to\infty$ and $s\to\infty$ with $s<T$ we have that $\hat
\theta(s)\to \lambda_1 R/B^2$ which is the 
optimal solution
on the infinite temporal horizon.
\rev{The transformation $\Phi$ defined above acts on the temporal domain $[0, T]$ and implements a \emph{reversal of time}, 
that we can also denote as $t \to T - t$. Given a trajectory on $[0, T]$, applying a time reversal transformation to it (as we did for the parameter $\theta$) entails considering the trajectory in which the direction of time is reversed. The dynamics that we observe while moving forward with this new temporal variable are the same dynamics
that we would observe when starting from $T$ and moving backward in the original variable. This comment should also give an
intuitive justification of why we could trade 
final conditions with initial conditions when this 
transformation is applied.}

\subsection{Neural costate estimation}\label{sec:costate-est}
The main contribution of this work is the proposal of a novel method to find a
forward approximation of the \emph{costate} trajectory by making use of an
additional Feed-forward Neural Network (FNN) to predict its values.  We assume that the costate $\costate$
is estimated by a FNN $\costatenet(\cdot,\cdot, \thetavar)
\colon\R^\statedim\times\R^\inpdim\to\R^\statedim$ with parameters $\thetavar\in\R^\mudim$ and then
we generalize steps 1--4 that we employed in the LQ problem in the previous subsection as follows:
\footnote{Here we have assumed, mainly to avoid unnecessary long equations,
that the $\costatenet(\cdot,\thetavar)$ take as input only the
state; however, more generally its domain
could also be enriched with the input signal $\inp$. Indeed, in the experimental section
we will show some case-studies where this is the case.}

\begin{enumerate}
\item We randomly initialize the parameters of the network
$\costatenet$ to the values $\theta(0)$  
and select an initial state $\state^0$. This allows us to compute
$\costatenet(\state^0,\theta(0))$ and in turn $\state'(0)$, using Hamilton equations
with $\costatenet(\state^0,\theta(0))$ in place of $p(0)$.\footnote{Due to the fact
that the controls enters in the state equation linearly (see Eq.~\eqref{eq:full-state-model})
if the Lagrangian is quadratic in the controls, like in Eq.~\eqref{eq:lagrangian-fcn-cp},
then the Hamiltonian~\eqref{eq:hamiltonian} can be computed in closed form.}

\item At a generic temporal instant $\time$, we assume to know $\state(\time)$ and $\theta(\time)$,
we compute
$\state'(t)=H_\costatevar(\state(\time),\costatenet(\state(t),\theta(\time)),\time)$
and define the loss function (see Remark~\ref{rem:loss-thetadot})
\begin{equation}\label{eq:loss-thetadot}
\begin{split}
\Omega_\time(\phi):=\frac{1}{2}\Vert & \costatenet_\statevar(\state(\time),\theta(t))\cdot\state'(\time)
+\costatenet_\thetavar(\state(\time),\theta(t))\cdot\phi\\
&{}+H_\statevar(\state(\time),\costatenet(\state(\time),\theta(t)),\time)\Vert^2+
\frac{\eps}{2}\Vert\phi\Vert^2.
\end{split}
\end{equation}
We choose $\delta\theta(t)\in\argmin_{\phi\in\R^\mudim} \Omega_t(\phi)$ by performing a 
gradient descent method on $\Omega_\time$.
\item  We numerically integrate the equation
\begin{equation}\label{eq:time-reversal-theta}
\theta'(\time)= -\delta\theta(t)
\end{equation}
with an explicit Euler step, in order to update the values of $\theta$.
We denote this equation (see Remark~\ref{rem:gen-riccati})
the \emph{ time-reversed generalized Riccati equation}.
\item Finally, we compute the control parameter using Eq.~\eqref{eq:syntesis-costate}
where the optimal costate is replaced with its estimation given
with the network $\costatenet$.
\end{enumerate}
\medskip

\begin{remark}
Notice that the assumption that the costate is computed as a function of the
state is consistent with its definition in terms of the value function
$\costate(\time)=D\valuef(\state(\time),\time)$. The only real assumption
that we are making is that the explicit temporal dependence in
$D\valuef(\state(\time),\time)$ is captured by the dynamic of the
parameters $\theta(\time)$ of the network $\costatenet$.
\end{remark}

\begin{remark}\label{rem:loss-thetadot}
The loss function $\Omega_t$ defined in Eq.~\eqref{eq:loss-thetadot} is
designed to enforce the consistency between the following two different
estimates of the temporal variations of the costate: \emph{i. } the one that
comes from the explicit temporal differentiation of $d
\costatenet(\state(\time), \theta(\time))/dt=
\costatenet_\statevar(\state(\time),\theta(t))\cdot\state'(\time)
+\costatenet_\thetavar(\state(\time),\theta(t))\cdot\theta'(t)$ and
\emph{ii. } the estimate
$-H_\statevar(\state(\time),\costatenet(\state(\time),\theta(t)),\time)$
obtained from the Hamilton equations.
\end{remark}

\begin{remark}
Eq.~\eqref{eq:time-reversal-theta} prescribes a dynamics for the
parameters $\theta$ that can be interpreted as a time reversal
transformation $t\mapsto T-t$ on the dynamics of the parameters of the
network $\costatenet$ induced by Hamilton equations (see
Remark~\ref{rem:loss-thetadot}). Our conjecture is that this
prescription implements a policy that induces stability to the
Hamilton equations (see Section~\ref{sec:flip}).
\end{remark}

\begin{remark}\label{rem:gen-riccati}
Notice that in the LQ case described in the previous section,
Eq.~\eqref{eq:time-reversal-theta} indeed reduces to
Eq.~\eqref{eq:time-reverse-riccati}.
Indeed, if $\mu(\statevar,\thetavar)=\thetavar\statevar$ and
$\eps\equiv0$ we have that for LQ
$\argmin_{\phi\in\R^\mudim} \Omega_\time(\phi)=\{(B^2/R)\theta^2(\time)
-2S\theta(\time)-Q\}$,
hence $\delta\theta(\time)=(B^2/R)\theta^2(\time)-2S\theta(\time)-Q$. In this
case the equation $\theta'(\time)=+\delta\theta(\time)$ would be exactly
the Riccati equation with the correct sign.
\end{remark}

\section{Experiments}
\label{sec:exp}
In the previous sections, we have presented our proposal within the framework of a continuous time setting. In the subsequent segment of our study, which is dedicated to experimentation, we employ explicit Euler steps of magnitude $\tau$ to approximate the differential equations. The number of time steps will be denoted as $n_T$. Moreover, we assume that the gradient descent procedure mentioned in Sec.~\ref{sec:costate-est} is characterized by a number of iterations $n_{\mathrm{iter}}$ and a learning rate $\lambda$. In appendix~B we report a summarizing algorithm of all the procedure presented so far.
In order to provide a proof-of-concept of the ideas of this paper, we analyze the capability of our forward-optimization procedure of solving three different tasks with neural estimators: (a) tracking a reference signal, (b) predicting the sign of an input signal and (c) classifying different wave-shapes provided as input signal. 
The experiences of this section are based on a shared definition of the Lagrangian function $\ell$ of Eq.~\eqref{eq:cost-functional}, which consists of a penalty term on the tracking quality of the target signal, referred to as $\target$, and regularization terms both on the outputs of the neurons in network $\statenet$ and on the velocities of its parameters, i.e., the control. Formally,
\begin{equation}\label{eq:lagrangian-fcn-cp}
\loss(\controlvar, \statevar, \timevar) = \frac{1}{2} q(\pi(\statevar) - \target(\timevar))^{2} + \frac{1}{2} r_{1} \sum_{i=\outdim}^{\statedim} \statevar_{i}^{2} + \frac{1}{2} r_{2} \sum_{i=0}^{\controldim} \controlvar_{i}^{2},
\end{equation}
where $\target(\timevar)$ is the task-specific target signal at time $\timevar$ and $q, r_{1}, r_{2} \ge 0$ are customizable constant coefficients. We recall that $\pi$ is a fixed map that, in this case, we assume to simply select one of the neurons in the output of $\gamma$ (i.e., the first one). Basically, minimizing the Lagrangian implies forcing the output of a neuron, $\pi(\statevar(\timevar))$, to reproduce the target signal for every $\timevar \in [\time_0, T]$. The goal of our experiences is to find the optimal control $\control$ which minimizes the cost functional defined in Eq.~\eqref{eq:cost-functional}. 
In the following subsections we report the results obtained for each experiment, where the initial time step is set to $\time_0=0$, the outputs of the neurons of $\statenet$ at the $\time_0=0$ are initialized to $0$, and the parameters of both networks $\statenet$ and $\costatenet$ start from random values. \rev{All the experiments 
have been conducted using Python 3.9 with PyTorch 2.0.0 on 
a Windows 10 Pro OS with an Intel Core i7 CPU and 16GB of memory.}

\begin{figure}[t]
\centerline{\includegraphics{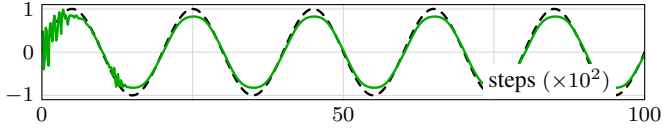}}
\caption{Tracking of a sinusoidal target signal using a recurrent network
$\gamma$ of 2 neurons. Black dashed line: target signal, continuous green
line: response of $\gamma$. The recurrent network has no input and the target
is a sine wave with frequency 0.001 Hz. \protect\label{fig:sine-tracking}%
}
\end{figure}

\paragraph{Case (a): tracking a target signal}
Let us consider the case where the target signal is given by
$\target(\timevar) = \sin(2\pi\varphi\timevar)$, where $\varphi=0.001$ Hz is the frequency of $\target$, and we want the recurrent network $\gamma$ to track it. Let us choose the model of the network $\gamma$ 
as composed of 
2 recurrent neurons fully connected to their inputs, $\neuron_0$ and $\neuron_1$, with $\tanh$ activation function, following Eq.~\eqref{eq:residual} (of course, in this experience there is no $\inp$). We also downscale the $-\neuron_i$ term by $0.5$. Moreover, we choose the network $\mu$ as a fully-connected feed-forward net, with 1 hidden layer made up of 20 neurons with ReLU activation functions. The output layer of $\mu$ has linear activation. 
With the choice of $\tau = 0.5$ s, $n_T=10^{4}$ time steps, $q=10^{4}$, $r_{1}=10^{3}$, $r_{2}=10^{5}$, we get the results plotted in Fig.~\ref{fig:sine-tracking}. The target signal is the black dashed line, the response of $\gamma$ is the continuous green line. 
The number of iterations for updating the derivatives of the weights of $\mu$ is set to  $n_{\mathrm{iter}}=100$, with a learning rate $\lambda=10^{-5}$ and a decay factor $\eps=10^{3}$. It is possible to see how the response of $\gamma$ is able to track the target signal and the accuracy of the tracking quickly improves in the early time steps. 
The amplitude reached by the response of $\statenet$ is affected by a slight reduction with respect to $\target$, due to the regularization terms in the Lagrangian function. This experiment confirms that the tracking information, provided through $\ell$, is able to induce appropriate changes in the parameters of $\statenet$, that allow such network to follow the signal. Notice that this signal propagates through the state-to-costate map $\costatenet$ and it is not directly attached to the neurons of $\statenet$, as in common machine learning problems.

\begin{figure}[t]
\centerline{\hfil\includegraphics{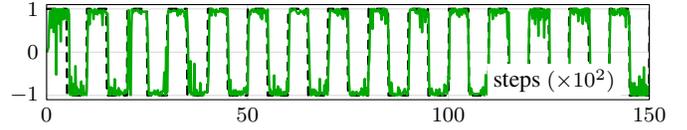}\hfil}
\caption{Sign prediction of a sinusoidal signal using a recurrent network
$\gamma$ of 2 neurons. Black dashed line: target signal, continuous green
line: response of $\gamma$. The input of the network is a sinusoidal wave
with frequency 0.002 Hz. The target to track is the sign of the input
signal.\protect\label{fig:sign-of-sine}}
\end{figure}

\paragraph{Case (b): predicting the sign of an input signal}
Let us now assume that both networks $\gamma$ and $\mu$ receive as input a sinusoidal signal $\inp(\timevar) = \sin(2\pi\varphi\timevar)$ with frequency $\varphi=0.002$ Hz. The task of predicting the sign of $\inp(\timevar)$ can be translated in a tracking control problem, where the target signal $\target(\timevar)$ is defined as
\[\target(\timevar) = 
\begin{cases}
			1, & \text{if $\inp(\timevar)\ge 0$}\\
            -1, & \text{otherwise}.
\end{cases}\]
Here, we choose the model of the network $\gamma$ as in Case (a), while the network $\mu$ has $\tanh$ activation functions in the hidden layer. With the choice of $\tau = 0.5$ s, $n_T=1.5 \times 10^{4}$ time steps, $q=10^{5}$, $r_{1}=10^{3}$, $r_{2}=10^{2}$, we get the results plotted in Fig.~\ref{fig:sign-of-sine}. The target signal is the black dashed line, the response of $\gamma$ is the continuous green line. 
The maximum number of iterations for updating the derivatives of the weights of $\mu$ is again set to $n_{\mathrm{iter}}=100$, with an adaptive learning rate $\lambda$ which starts from $10^{-3}$ and a decay factor $\eps=10^{4}$. Here, the adaptive strategy for $\lambda$ is the one used by the Adam optimizer. This task is clearly more challenging than the previous one, since we ask $\gamma$ to react in function of the input $\inp$, still using the state-to-costate map $\costatenet$ as a bridge to carry the information.
Interestingly, also in this case, the response of $\statenet$ is able to track the target signal, correctly interleaving the information from the previous state and the current input.

\begin{figure}[t]
\centerline{\hfil\includegraphics{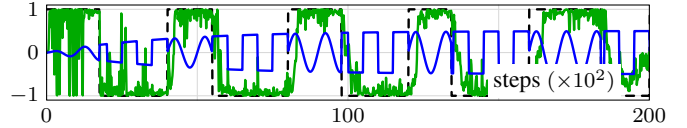}\hfil}
\caption{Classification of wave-shapes using a recurrent network $\gamma$ of
2 neurons. Black dashed line: target signal, continuous green line: response
of $\gamma$, continuous blue line: input signal. The input of the network is
a sequence of sines and square waves, multiplied by a smoothing factor $1 -
\exp(-\timevar/\psi)$, where $\psi = 2000\,s^{-1}$.
\protect\label{fig:wave-class}}
\end{figure}

\paragraph{Case (c): classifying the different wave-shapes of an input signal}
Finally, we consider the case in which both networks $\statenet$ and $\mu$ get as input a piece-wise defined signal characterized by two different wave-shapes. More precisely, we assume that $\inp(\timevar) = \inp_{1}(\timevar) = (1/2) \sin(2\pi\varphi\timevar)$ or $\inp(\timevar) = \inp_{2}(\timevar) = -(1/2) (-1)^{\lfloor 2\varphi\timevar \rfloor}$, with frequency $\varphi=0.002$ Hz, in different time intervals randomly sampled on the whole time horizon. Moreover, we multiply $\inp(\timevar)$ by a smoothing factor $1 -
\exp(-\timevar/\psi)$, where $\psi = 2000\,s^{-1}$, in order to help the network $\mu$ learning to estimate the costate. The task of classifying the wave-shape of $\inp(\timevar)$ can be again translated in a tracking control problem, where the target signal $\target(\timevar)$ is defined as
\[\target(\timevar) = 
\begin{cases}
			1, & \text{if $\inp(\timevar)=\inp_{1}(\timevar)$}\\
            -1, & \text{if $\inp(\timevar)=\inp_{2}(\timevar)$}.
\end{cases}\]
In this case, the networks have to deal with the need of differently reacting in different time spans.
We choose the models of the networks $\gamma$ and $\mu$ as in Case 
(b). The maximum number of iterations for updating the derivatives of the weights of $\mu$ is again set to $n_{\mathrm{iter}}=100$, with an adaptive learning rate $\lambda$ which starts from $10^{-3}$ and a decay factor $\eps=10^{4}$. The adaptive strategy for $\lambda$ is the same as in Case(b). With the choice of $\tau = 0.5$ s, $n_T=2 \times 10^{4}$ time steps, $q=10^{5}$, $r_{1}=10^{3}$, $r_{2}=10^{2}$, we get the results plotted in Fig.~\ref{fig:wave-class}. 
The target signal is the black dashed line, the response of $\gamma$ is the continuous green line, the input signal is the continuous blue line. 
 Also in this case, the response of $\gamma$ is able to track the target signal, even if we experienced a small delay in the tracking process that we believe to be motivated by the need of smoothly updating the state, in order to favour the transition in switching from predicting $-1$ to $1$ and vice-versa.
In Fig.~\ref{fig:int-lag} we report the average value of the Lagrangian for all the tasks that we exposed, obtained dividing the integral of $\ell$ in $[0, s]$ by $s$, for each $s \in (0, T]$. It is possible to notice that the mean value of the Lagrangian function decreases as time goes by, reflecting the improvement of the model in tracking the different target signals.

\begin{figure}
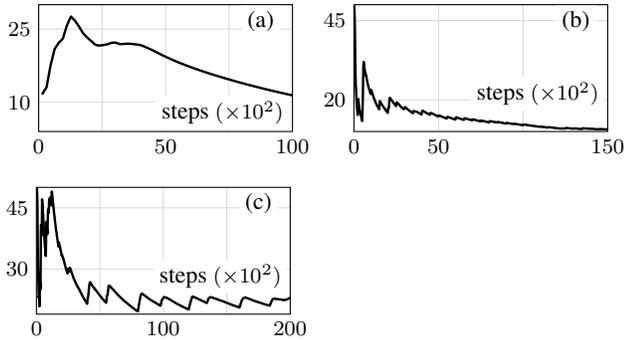

    \centering
\centerline{\includegraphics{./figs/figs-5.mps}\hfil
\includegraphics{./figs/figs-6.mps}}
\medskip
\hbox to\hsize{\includegraphics{./figs/figs-7.mps}\hfil}
    \caption{Average value of the Lagrangian function for Case (a), (b) and (c).}
    \label{fig:int-lag}
\end{figure}


\section{Conclusions and future work}
\label{sec:lim}
This paper introduced a novel theory of optimization that points out a new perspective in the field of optimal control. The forward-in-time Hamiltonian optimization opens up new possibilities for real-time adaptation, tracking and control in lifelong learning scenarios.
By bridging the gap between optimal control and deep learning, this innovative methodology paves the way for significant advancements in the learning and adaptation capabilities of autonomous systems in dynamic environments. The paper delved into the theoretical foundations of forward-in-time Hamiltonian optimization, with a particular emphasis on the concept of time-reversed generalized Riccati equation. 
Future research will focus on enhancing the learning capabilities of the model to facilitate its application in lifelong learning tasks.

In our experiments we have shown that our proposal can be used to efficiently solve different kinds of tracking control problems, where the target signal is always present for each time step. It is important to emphasize that the model is contingent upon a considerable number of parameters and exhibits a high degree of sensitivity to their variations. Consequently, tuning these parameters can be challenging. We recall that gaining explicit generalization capabilities (i.e., when the target signal is not given) is not a goal pursued within the scope of this study, but it will be the main point of our future work. Indeed, we mention that this novel approach still has limitations in such a direction. In most of the experiments, the response of $\statenet$ is not able to generate the target signal if we mask it after a certain number of time steps, freezing the weights of $\mu$ up to the time horizon. An example of this behavior can be seen in Fig.~\ref{fig:wave-class-gen}.
However, in Case (b) we have registered that $\gamma$ is able to reproduce the target $\target$ even if it is masked after 15000 steps, as shown in Fig.~\ref{fig:sign-of-sine-gen}.
This result suggests further investigations on this direction and we will orient our research interest on the capability of learning of our proposal, in order to apply it to lifelong learning tasks \cite{de2021continual}.

\begin{figure}
\centerline{\hfil\includegraphics{./figs/figs-3.mps}\hfil}
\caption{Example of lack of generalization of our approach in the wave-shape
classification task.  Black dashed line: target signal, continuous green
line: response of $\gamma$, continuous blue line: input signal. The target
signal is provided up to 20000 steps and it is masked up to the time
horizon.\protect\label{fig:wave-class-gen}}
\centerline{\hfil\includegraphics{./figs/figs-4.mps}\hfil}
\caption{Example of generalization in the sign prediction task. Black dashed
line: target signal, continuous green line: response of $\gamma$. The target
is given up to 15000 steps.\protect\label{fig:sign-of-sine-gen}} 
\end{figure}

\section{Acknowledgments}
This work has been supported by the French government, through the 3IA C\^ote d’Azur, Investment in the Future, project managed by the National Research Agency (ANR) with the reference number ANR-19-P3IA-0002.

\bibliography{aaai24}

\appendix

\section{Proof of Theorem~\ref{th:HJB}}\label{appendix:theo}
Let $\timevar\in[\time_0,T)$ and $\statevar\in\R^\statedim$.
Furthermore, instead of the optimal
control let us use a constant control $\alpha_1(\time)=\controlvar \in \R^\controldim$ for times $\time\in[\timevar,\timevar+\epsilon]$
and then the optimal control for the remaining temporal interval.
More precisely, following the notation introduced in Definition~\ref{def:solution},
let us pose
\[\alpha_2\in\argmin_{\alpha\in\controlspace} C_{\state(\timevar+\eps; \controlvar,\statevar,
\timevar),\timevar+\eps}(\alpha).\]
Now consider the following control
\begin{equation}
\alpha_3(\time)=
\begin{cases}
\alpha_1(\time) &  \text{if} \ \time\in[\timevar,\timevar+\eps)\\
\alpha_2(\time) &  \text{if} \ \time\in[\timevar+\eps,T] \ .
\end{cases}
\end{equation}
Then the cost associated to this control is
\begin{equation}
\begin{split}
C_{\statevar,\timevar}(\alpha_3)&=\int_{\timevar}^{\timevar+\eps} \ell(a, \state(\time;a,
\statevar,\timevar),\time)\,d\time \\
&+\int_{\timevar+\eps}^T \ell(\alpha_2(\time), \state(\time;\alpha_2,\statevar,
\timevar),\time)\,ds \\
&= \int_{\timevar}^{\timevar+\eps} \ell(a, \state(\time;a,
\statevar,\timevar),\time)\,d\time \\
&+ v(\state(\timevar+\eps; a,\statevar,\timevar),\timevar+\eps)
\end{split}
\end{equation}
By definition of value function we also have that $v(\statevar,\timevar)\le
C_{\statevar,\timevar}(\alpha_3)$. When  rearranging this inequality, dividing by
$\eps$, and making use of the above relation we have
\begin{equation}
\begin{split}
& \frac{v(\state(\timevar+\eps; a,\statevar,\timevar),\timevar+\eps)-v(\statevar,\timevar)}{\eps}+ \\
& \frac{1}{\eps}\int_\timevar^{\timevar+\eps} \ell(a, \state(\time;\controlvar,\statevar,\timevar),\time)\,d\time\ge0
\end{split}
\end{equation}
Now taking the limit as $\eps\to 0$ and making use of the fact that
$\state'(\timevar,\controlvar,\statevar,\timevar)=f(\statevar,\controlvar,\timevar)$
we get
\begin{equation}
v_\timevar(\statevar,\timevar)+ Dv(\statevar,\timevar)\cdot 
f(\statevar,\controlvar,\timevar)+\ell(\controlvar,\statevar,\timevar)\ge0.
\end{equation}
Since this inequality holds for any chosen $\controlvar\in\R^\controldim$ we can say that
\begin{equation}
\inf_{\controlvar\in \R^{\controldim}} \{v_\timevar(\statevar,\timevar)+ Dv(\statevar,\timevar)\cdot 
f(\statevar,\controlvar,\timevar)+\ell(\controlvar,\statevar,\timevar)\}\ge 0
\end{equation}
Now we show that the $\inf$ is actually a $\min$ and, moreover, that minimum is
$0$. To do this we simply choose
$\alpha^*\in\argmin_{\alpha\in\controlspace} C_{\statevar,\timevar}(\alpha)$ and denote
$a^*:=\alpha^*(\timevar)$, then 
\begin{equation}
\begin{split}
v(\statevar,\timevar)&= \int_\timevar^{\timevar+\eps} \ell(\alpha^*(\time), \state(\time;
\control^*,\statevar,\timevar),\time)\,d\time \\
&+v(\state(\timevar+\eps;\control^*,\statevar,\timevar).
\end{split}
\end{equation}
Then again dividing by $\eps$ and using that
$\state'(\timevar;\control^*,\statevar,\controlvar)=f(\statevar,\controlvar^*,\timevar)$
we finally get
\begin{equation}
v_\timevar(\statevar,\timevar)+ Dv(\statevar,\timevar)\cdot f(\statevar,
\controlvar^*,\timevar)+\ell(\controlvar^*,\statevar,\timevar)=0
\end{equation}
But since $a^*\in \R^{\controldim}$ and we knew that
$\inf_{\controlvar\in \R^{\controldim}} \{v_\timevar(\statevar,\timevar)+ Dv(\statevar,\timevar)\cdot 
f(\statevar,\controlvar,\timevar)+\ell(\controlvar,\statevar,\timevar)\}\ge 0$
it means that
\begin{align}
\begin{split}
&\inf_{a\in\R^\controldim} \{v_\timevar(\statevar,\timevar)
+ Dv(\statevar,\timevar)\cdot f(\statevar,\controlvar,\timevar)
+\ell(\controlvar,\statevar,\timevar)\}
= \\
&\min_{a\in\R^\controldim} \{v_\timevar(\controlvar,\timevar)+ Dv(\statevar,\timevar)\cdot f(\statevar,\controlvar,\timevar)+\ell(\controlvar,\statevar,\timevar\}=0.
\end{split}
\end{align}
Recalling the definition of $H$ we immediately see that
the last inequality is exactly (HJB). 

\section{Main Algorithm}\label{appendix:alg}
We report in Alg.~\ref{alg:main}  the main algorithm described in Section~\ref{sec:costate-est}.
We omit to explicitly recall the dependence of $\state, \costate, \neuron, \control, \inp, \theta, w, b, k$ (and their derivatives) on $t$ treating those as local variable whose value is updated at each new temporal
instant.
The algorithm is structured as described below.
Firstly, we provide the initialization of the variables $\neuron(0)$ (the initial values of the outputs of  neurons in network $\gamma$), $w(0)$, $b(0)$ and $k(0)$ (the initial weights of network $\gamma$), $\theta(0)$ (weights of networks $\mu$), $\theta'(0)$ (the derivative of $\theta$). Then, 
the form of the Lagrangian $\ell$ must be provided, 
together with those hyper-parameters that define the optimization process, namely the learning rate $\lambda$, the number of iterations $n_{\mathrm{iter}}$ for the gradient descent loop to update 
$\theta'$, the magnitude $\tau$ of the Euler step. The input $\inp$, if present, 
is provided at each time step,
while we indicate with $n_T$ the total number of time steps. 
\begin{algorithm}[t!]
\caption{Main Algorithm.} \label{alg:main}
\begin{algorithmic}  
    \Require $\neuron(0)=(0,0,\dots,0)$ (initial state of the neurons see Section~\eqref{sec:ocp-net}), 
    $w(0)$, $b(0)$, $k(0)$ (initial weights, randomly initialized); $\theta(0)$ (initial weights of $\mu$, randomly initialized); $\theta'(0)$ (initialized at zero); $\lambda$ (learning rate 
    for the optimization of Eq.~\eqref{eq:loss-thetadot}); $n_{\mathrm{iter}}$ 
    (number of iterations for the optimization of Eq.~\eqref{eq:loss-thetadot}); $\tau$ (magnitude of Euler step); $\ell$ (Lagrangian function); $\inp$ (input signal, if present); $n_T$ (total number of time steps).
    \vskip 5mm
    \For{$\time = 0, \dots, n_T -1$}
        \State $\costate \gets \mu (\state, \inp, \theta)$ \Comment{Compute $\costate$ with network $\mu$}
        \State $\mathcal{H}(\state, \costate, \controlvar, \time) \gets \ell(\controlvar, \state, \time) + \costate \cdot \state'$ \Comment{Compute $\mathcal{H}$}
        \State $\control^* \gets  \argmin_{\controlvar \in \controlspace} \mathcal{H}(\state, \costate, \controlvar, \time)$ \Comment{Compute $\control^*$, ~\eqref{eq:opt-ctrl}}
        \State $\neuron' \gets \gamma (\neuron, \inp, w, b, k)$ \Comment{Compute $\neuron'$, ~\eqref{eq:full-state-model}}
        \State $(w', b', k') \gets \control^*$ \Comment{Assign $(w', b', k')$, ~\eqref{eq:full-state-model}}
        \State $\state' \gets (\neuron', w', b', k')$ \Comment{Construct $\state'$}
        \State$H(\state, \costate, \time) \gets \ell(\control^*, \state, \time) + \costate \cdot \state'$ \Comment{Compute $H$, ~\eqref{eq:hamiltonian}}
        \State $\costate' \gets -H_\statevar(\state, \costate, \time)$ \Comment{H-equation for $\costate$, Eq.~\eqref{eq:H-J-equation}}
        \State $\mathrm{iter} \gets 0$
        \While{$\mathrm{iter} < n_{\mathrm{iter}}$} \Comment{Gradient descent loop}
            \State $\Omega_\time(\theta') \gets \hbox{Eq.~\eqref{eq:loss-thetadot}} $ \Comment{Loss function for $\theta'$}
            \State $\theta' \gets \theta' - \lambda \nabla_{\theta'} \Omega_\time(\theta')$ \Comment{Update $\theta'$}
        \State $\mathrm{iter} \gets \mathrm{iter} + 1$
        \EndWhile
    \State $\theta \gets \theta - \tau \theta'$ \Comment{\textsl{time-reversed generalized Riccati equation}, Eq.~\eqref{eq:time-reversal-theta}}
    \State $\state = \state + \tau \state'$ \Comment{Update $\state$ via Euler Step}
    \State $\time \gets \time+1$
    \EndFor
\end{algorithmic}
\end{algorithm}

Then, for every time step $\time=0, \dots, n_T$, we proceed as follows. Firstly, we compute an approximation of the costate $\costate(\time)$ using the network $\costatenet$, which takes as input the state $\state(\time)$ and the input $\inp(\time)$, if present.

When the components of $f$ are defined as in Eq.~\eqref{eq:full-state-model} and $\ell$ has the form in Eq.~\eqref{eq:lagrangian-fcn-cp}
the computation of the optimal control, given by Eq.~\eqref{eq:syntesis-costate}, can be done 
explicitly. Let us name $\mathcal{H}$ the function of which we want to take the $\argmin$ in
Eq.~\eqref{eq:syntesis-costate}, i.e. let us define
for every $\statevar \in \R^{\statedim}, \costatevar \in \R^{\statedim}, \controlvar \in \R^{\controldim}, \timevar \in [t_0, T]$:
\begin{equation}\label{eq:pseudo-hamiltonian}
 \mathcal{H}(\statevar,\costatevar, \controlvar, \timevar):=\costatevar\cdot f(\statevar ,\controlvar,\timevar)+\ell(\controlvar,\statevar,\timevar)   
\end{equation}
Furthermore, let us denote with $\costate^{w}_{ij}$,  $\costate^{b}_{i}$, $\costate^{k}_{ij}$ the costates associated to $w_{ij}$, $b_{i}$, $k_{ij}$, respectively.\footnote{Here with "associated to" we mean 
that $\costate^{w}_{ij}$,  $\costate^{b}_{i}$, $\costate^{k}_{ij}$ are the components of the costate obtained by differentiating the Hamiltonian with respect to
the corresponding component of the state.}
Then the controls can be obtained from the stationarity conditions of $\mathcal{H}$, i.e.
$\mathcal H_\controlvar=0$ that can be  factorized as follows:
\[\begin{cases}
r_2 \omega_{ij} + p^{w}_{ij}=0,& \text{for }(j,i)\in\arcs; \\
r_2 \nu_{i} + p^{b}_{i},&  \text{for } i\in\vertices=0; \\
r_2 \chi_{ij} + p^{k}_{ij}=0,&  \text{for } i\in\vertices \hbox{ and }
j=1,\dots,\inpdim,
\end{cases}\]
from which we immediately get the explicit formulas for the components $\omega^*_{ij}$, $\nu^*_{i}$ and $\chi^*_{ij}$ of the optimal control $\control^*$
\begin{equation}\label{eq:opt-ctrl}
\begin{cases}
\omega^*_{ij} = -p^{w}_{ij}/r_2 &(j,i)\in\arcs\\
\nu^*_{i} = -p^{b}_{i}/r_2 &  i\in\vertices\\
\chi^*_{ij} = -p^{k}_{ij}/r_2 & i\in\vertices \hbox{ and }
j=1,\dots,\inpdim.
\end{cases}
\end{equation}

Using the net $\statenet$, we then compute $\neuron'(\time)$, while using the optimal control 
$\alpha^*(\time)$
and Eq.~\eqref{eq:full-state-model} we compute $(w'(\time), b'(\time), k'(\time))=\alpha^*(\time)$. 
We have therefore obtained the whole $\state'(\time) = (\neuron'(\time), w'(\time), b'(\time), k'(\time))$. 

Recalling Eq.~\eqref{eq:hamiltonian} and using the optimal controls, we compute the Hamiltonian as $H(\state(\time), \costate(\time), \time) = \ell(\control^*(\time), \state(\time), \time) + \costate(\time) \cdot \state'(\time)$, from which we obtain the derivative $\costate'(\time)$ of the costate by Eq.~\eqref{eq:H-J-equation}. 

Finally, we iteratively update the derivatives $\theta'(\time)$ of the weights of $\costatenet$ by gradient descent, with learning rate $\lambda$, optimizing the loss function $\Omega_\time(\theta')$ defined in Eq.~\eqref{eq:loss-thetadot}. The procedure terminates after $n_{\mathrm{iter}}$ iterations, providing a new estimation of $\theta'$. Integrating the time-reversed generalized Riccati Eq.~\eqref{eq:time-reversal-theta}, we update the values $\theta(\time)$ of the weights of $\costatenet$ and, with an Euler Step of magnitude $\tau$, we update the state $\state$.

\end{document}